\documentclass[12pt]{amsart}

\usepackage{xy, graphicx, color,hyperref,verbatim,cleveref,multirow}
\pdfoutput=1
\usepackage{amsfonts, ytableau,array}
\usepackage[utf8]{inputenc}
\usepackage{amssymb,amsmath,amsthm,bm,mathrsfs,enumerate,bbm,enumitem}
\usepackage[a4paper, total={6in, 8in}]{geometry}
\usepackage[english]{babel}
\usepackage{fancyhdr}
\usepackage[utf8]{inputenc}
\usepackage{mathtools,cite,mathdots}
\usepackage{tikz}
\usepackage{tikz-cd}
\usetikzlibrary{matrix}
\usepackage{booktabs}
\usepackage{blindtext}
\usepackage{array}
\newcolumntype{P}[1]{>{\centering\arraybackslash}p{#1}}
\PassOptionsToPackage{linktocpage}{hyperref}
   \hypersetup{
           breaklinks=true,  
           colorlinks=true,  
           pdfusetitle=true, 
           citecolor={blue!80!black}
        }

\xyoption{all}

\newenvironment{roster}
 {\begin{enumerate}[font=\upshape,label=(\roman*)]}
 {\end{enumerate}}

 \newtheorem{thm}{Theorem}[section]
\newtheorem{cor}[thm]{Corollary}
\newtheorem{lemma}[thm]{Lemma}
\newtheorem{prop}[thm]{Proposition}

\theoremstyle{definition}
\newtheorem{ex}[thm]{Example}
\theoremstyle{definition}

\theoremstyle{definition}

\newcommand{\nc}{\newcommand}

\nc{\mc}{\mathcal}
\nc{\mb}{\mathbb}
\nc{\mf}{\mathfrak}
\nc{\ul}{\underline}
\nc{\ol}{\overline}
\nc{\N}{\mb N}
\nc{\R}{\mb R}
\nc{\Z}{\mb Z}
\nc{\Q}{\mb Q}
\nc{\C}{\mb C}
\nc{\F}{\mb F}
\nc{\dmo}{\DeclareMathOperator}

\nc{\mat}[4]{
    \begin{pmatrix}
      #1 & #2 \\
      #3 & #4
    \end{pmatrix}
}
\dmo{\Ker}{Ker} \dmo{\val}{val} \dmo{\ord}{ord}
\dmo{\I}{I}
\dmo{\II}{II}
\dmo{\odd}{odd}
\dmo{\sgn}{sgn}
\dmo{\dett}{det}
\dmo{\Span}{Span}
\dmo{\Syl}{Syl}
\dmo{\diag}{diag}
\dmo{\Sq}{Sq}
\dmo{\irr}{Irr}
\nc{\beq}{\begin{equation*}}
\nc{\eeq}{\end{equation*}}
\nc{\half}{\frac{1}{2}}
\dmo{\Ima}{Im}
\dmo{\too}{top}
\dmo{\mult}{mult}
\dmo{\Mod}{mod}

\dmo{\core}{core}
\dmo{\res}{res}
\dmo{\lin}{lin}
 
 \dmo{\St}{St}
 \dmo{\st}{st}
 
 \dmo{\Tr}{Tr}
 \dmo{\RO}{RO}
\dmo{\Sp}{Sp}
\dmo{\SO}{SO}
\dmo{\SL}{SL}
\dmo{\GL}{GL}
 \dmo{\Spin}{Spin}
\dmo{\GSp}{GSp}
\nc{\la}{\lambda}
  \nc{\eps}{\varepsilon}
  
 \nc{\lip}{\langle}
 \nc{\rip}{\rangle}
\nc{\gm}{\gamma}

\dmo{\Perm}{Perm}
\dmo{\Res}{Res}
\dmo{\Ind}{Ind}
\dmo{\ind}{ind}
\dmo{\tr}{tr}
\dmo{\Sym}{Sym}
\dmo{\reg}{reg}
\dmo{\End}{End}
\dmo{\Hom}{Hom}
\dmo{\Pin}{Pin}
\dmo{\Or}{O}
\dmo{\SW}{SW}
\dmo{\orb}{orb}
\title[Stiefel-Whitney Classes]{Stiefel-Whitney Classes of representations of Dihedral Groups} 
\author{Sujeet Bhalerao}
\author{Rohit Joshi}
\author{Neha Malik}

\address{Department of Mathematics, University of Illinois Urbana-Champaign, IL, USA} \email{sujeetbhalerao@gmail.com}

\address{Indian Institute of Science Education and Research, Pune, India} \email{rohitsj@students.iiserpune.ac.in}

\address{Chennai Mathematical Institute, India} \email{nehamalik@cmi.ac.in}  

\keywords{Stiefel-Whitney classes, dihedral groups, lifting representations}
\subjclass{Primary 20G40, Secondary 55R40}

\DeclareMathOperator{\pr}{pr}

\dmo{\ORep}{ORep}
\dmo{\Rep}{Rep}

\usepackage{blindtext,graphicx}
\usepackage[absolute]{textpos}
\newcommand{\as}{\alpha}

\newcommand{\z}{\mathbb{Z}}

\newcommand{\cc}{\mathbb{C}}

\begin{document}

\maketitle
\tableofcontents
\begin{abstract}
We compute the Stiefel-Whitney Classes for representations of dihedral groups $D_{m}$ in terms of character values of order two elements. We also provide criteria to identify representations $V$ which lift to the double covers of the orthogonal group $\Or(V)$  and those with non-trivial mod $2$ Euler class.
\end{abstract}

\section{Introduction}
Let $G$ be a finite group, and $\pi$ be an orthogonal representation of $G$. To $\pi$, one can associate cohomology classes
$w_i(\pi)$, living in $H^i(G,\z/2\z)$, called \textit{Stiefel-Whitney Classes} (SWCs) of $\pi$. Their sum 
$w(\pi)=w_0(\pi)+w_1(\pi)+\hdots$
is known as the \textit{total SWC} of $\pi$. These classes are of importance in studying group cohomology through the representations of the group. We refer to \cite{GKT}, \cite[Section 2.3]{NSSL2} for more details.

The paper \cite{Ganguly} of Ganguly and Spallone computed the second SWC to characterize spinorial representations of symmetric groups. This led to a program of calculating the total SWCs of representations in terms of character values for various groups. Joshi-Ganguly completed the case of $\GL(n,q)$ for $q$ odd in \cite{GJgln}.  The SWCs for $\SL(2,q)$ were computed in \cite{NSSL2} by Malik-Spallone. In this paper, we describe SWCs of representations of dihedral groups in terms of character values. This is the first step in the calculation of SWCs for $\Or(n,q)$, since $\Or(2,q)$ is a dihedral group when $q$ is odd. An important tool in our calculation is the cohomological \emph{detection} of a group by its subgroups.

We write $H^*(G)$ for $H^*(G,\z/2\z).$ The subgroups $K_1,K_2$  of $G$ are said to detect the mod $2$ cohomology of $G$ if the restriction map
$$H^*(G)\to H^*(K_1)\oplus H^*(K_2)$$
is an injection. For instance, it is well known \cite[Chapter II, Corollary 5.2]{milgram} that a Sylow $2$-subgroup $K$ detects the mod $2$ cohomology of $G$, meaning the restriction $H^*(G)$ to $H^*(K)$ is injective. 

Write $C_n$ for the cyclic group of order $n$. Let $D_m$ be the dihedral group of order $2m$ with $`r$' the rotation by angle $2\pi/m$ and a reflection $`s$' as its generators.  All representations of $D_m$ are orthogonal. The calculation of SWCs for abelian dihedral groups $D_1$, $D_2$ is straightforward. We review them in Section \ref{small groups}. When $m$ is a multiple of $4$, from \cite{VS}, the mod $2$ cohomology of $D_m$ is

$$H^*(D_{m}, \Z/2\Z)=\frac{\Z/2\Z[x,y,w]}{(y^2+xy)},  $$
where $x,y$ are certain elements of degree $1$ and $w$ is of degree $2$. In this case, the detection of $H^*(D_m)$ by two Klein-$4$ groups is known (see \cite[Proposition 3.3]{FP} for instance or Section \ref{detect} below). We use this detection to get:

\begin{thm}\label{main}
	Let $m$ be a multiple of $4$, and $G=D_m$. Let $\pi$ be a representation of $G$. Then, 
	$$w(\pi)=(1+y)^{a_\pi}(1+x+y)^{b_\pi}(1+x+w)^{c_\pi},$$
	where 
	\begin{align*}
		a_\pi&=\frac14\big(\chi_\pi(1)-2\chi_\pi(rs)+\chi_\pi(r_c)\big),\\
		b_\pi&=\frac14\big(\chi_\pi(1)-2\chi_\pi(s)+\chi_\pi(r_c)\big), \text{ and }\\
		c_\pi&=\frac14\big(\chi_\pi(1)-\chi_\pi(r_c)\big).
	\end{align*}
\end{thm}
\noindent Here $r_c=r^{m/2}$, and $\chi_\pi(g)$ is the character value of the representation $\pi$ at $g\in G$.   \\

There are several corollaries of Theorem \ref{main}. When $m$ is a multiple of $4$, we first have:
\begin{cor}\label{wtriv}
A representation $\pi$ of $D_m$ is trivial if and only if $w(\pi)=1$.
\end{cor}

 Let $W(G)$ be the subgroup of \emph{complete} cohomology ring $H^\bullet(G)$, generated by the total SWCs of orthogonal representations of $G$ (see \cite[Section 2.6]{NSSL2}). For $G=D_m$, we obtain:

\begin{cor}\label{W(G)iso} 
The group $W(D_m)$ is free abelian with generators $1+y,1+x+y, 1+x+w.$
\end{cor}

For a representation $\pi$ of degree $d$, we define the top SWC $w_{\too}(\pi) := w_d(\pi)$.  When $\det\pi=1$, there is another relevant characteristic class $e(\pi)\in H^d(G,\z)$, called the \emph{Euler class} of $\pi$. From \cite[Property 9.5]{Milnor},  $w_{\too}(\pi)$ is the reduction of $e(\pi)$ mod $2$. Here, we describe $\pi$ with $w_{\too}(\pi)\neq 0$ for $m=0\pmod 4$:

\begin{cor}\label{topswc} 
The top SWC of $\pi$ is non-zero iff either $\ell_\pi=0$ or $m_\pi=0$ where
\begin{align*}
\ell_\pi&=\chi_\pi(1)+2\chi_\pi(s)+\chi_\pi(r_c),\\
m_\pi&=\chi_\pi(1)+2\chi_\pi(sr_c)+\chi_\pi(r_c).
\end{align*}
\end{cor}

For a complex vector space $V$, there are four double covers of the orthogonal group $\Or(V)$, namely $\Or(V)\times \z/2\z$, $\widetilde\Or(V)$, $\Pin^+(V)$ and $\Pin^-(V)$ (see \cite[Section 12]{SJext} for instance). We say an orthogonal representation $\pi$ of a group $G$ is \textit{spinorial} provided it lifts to $\Pin^-(V)$, as in \cite{JSP}. There is a cohomological criterion to see when $\pi$ lifts to the extensions above. For instance, $\pi$ always lifts to $\Or(V)\times \z/2\z$. We review this in Section \ref{spindef}.

When $4$ divides $m$, we obtain the following lifting criteria for the representations of $D_m$ in terms of character values:

\begin{cor}\label{spin}
A representation $(\pi,V)$ of $D_m$ lifts to:
\begin{roster}
\item $\widetilde \Or(V)$ iff  $a_\pi= b_\pi= c_\pi\pmod 2,$

\item $\Pin^+(V)$ iff  $c_\pi$ is even and       $$
	\binom{a_\pi}{2}= \binom{b_\pi}{2}= \binom{c_\pi }{2}\pmod 2
	,$$

\item $\Pin^-(V)$ iff $c_\pi$ is even and                 $$
	\binom{a_\pi+1}{2}= \binom{b_\pi+1}{2}= \binom{c_\pi + 1}{2}\pmod 2.
	$$
\end{roster}
\end{cor}
We also have such results for dihedral groups $D_m$ when $m$ is odd or $m=2\pmod 4$, proved in later sections.\\

 Let $G, G'$ be finite groups. Let $\pi, \pi'$ be orthogonal representations of $G$ and $G'$ respectively with $\deg \pi=d$, $\deg\pi'=d'$. Then $\Pi=\pi\boxtimes\pi'$, the external tensor product, is an orthogonal representation of $G\times G'$. We have:

\begin{thm}\label{prodDm}
 The representation $\Pi$ is spinorial if and only if both of the following hold:
\begin{roster}
\item The restriction of $\Pi$ to each of $G\times 1$ and $1\times G'$ is spinorial, and
\item $(dd'+1)$ is even or $\det\pi=1$ or $\det\pi'=1$.
\end{roster}
\end{thm}
When $G=D_{m}$ and $G'= D_{m'}$, it leads to:

\begin{cor}\label{aspin}
Suppose $m,m'$ both are multiples of $4$. All non-trivial irreducible representations of $D_{m}\times D_{m'}$ are aspinorial.

\end{cor}
This paper is organized as follows. In section \ref{notpre}, we set up notation and review the group cohomology of $D_{m}$. In section \ref{small groups}, we review the calculation of SWCs for $C_2$ and obtain SWCs for representations of $C_2 \times C_2$. In section \ref{D2m}, we prove Theorem \ref{main} and its Corollaries \ref{wtriv}, \ref{W(G)iso}.  Sections \ref{TSWCsec}, \ref{liftOV} are dedicated to proving Corollaries \ref{topswc}, \ref{spin} along with similar results for $D_m$ when $m$ is odd or $m=2\pmod 4$. In the last section, we give a general spinoriality result for the products of groups, leading to the proofs of Theorem \ref{prodDm} and Corollary \ref{aspin}.

{\bf Acknowledgments:} The authors would like to thank Steven Spallone for helpful conversations. The contribution of the first author comes from his MS Thesis, at the Indian Institute of Science Education and Research (IISER) Pune, India. The second author of this paper is a post doctoral fellow at IISER Pune, India and was supported by National Board of Higher Mathematics fellowship, India. The third author gratefully acknowledges the support of Institute fellowships from IISER Mohali and Chennai Mathematical Institute, India, during this work.

\section{Notation and Preliminaries}\label{notpre}

 Let $D_{m}$ be the dihedral group of order $2m $ with  presentation
$$D_{m}=\langle r,s\mid r^{m}=1, s^2=1, rs=sr^{-1}\rangle.$$

 We now review the representation theory of $D_{m}$ (see \cite[Section 5.3]{serre1977linear} for instance).

\subsection{Irreducible Representations of $D_{m}$} 
Write $`\sgn$' for the non-trivial linear character of $D_1$.  

For $m $ even, the group $D_{m}$ has four linear characters, namely $1$, $\chi_s$, $\chi_r$, $\chi_{rs}$:
\begin{align*}
1:(r,s)&\mapsto (1,1)\\
\chi_s:(r,s)&\mapsto (1,-1)\\
\chi_r:(r,s)&\mapsto (-1,1)\\
\chi_{rs}:(r,s)&\mapsto (-1,-1).
\end{align*}

Set $\theta_k=\frac{2\pi k}{m}$. 
There are also irreducible $2$-dimensional representations enumerated by $k=1,2,\hdots,m/2-1$:
$$\sigma_k:D_{m}\to \GL(2,\cc)$$ is given by $$\sigma_k(r)=\begin{pmatrix}
	\cos\theta_k&&-\sin \theta_k\\[2mm]
	\sin\theta_k&&\cos\theta_k
\end{pmatrix}\quad,\quad\sigma_k(s)=\begin{pmatrix}
	0&&1\\
	1&&0
\end{pmatrix}.$$

Write $\sigma=\sigma_1$; this is the standard $2$-dimensional representation of $D_{m}$.

When $m$ is odd, $D_{m}$ has two linear characters $1$, and $\chi_s$. Also, the representations $\sigma_k$ for $k=1,\hdots,(m-1)/2$ defined above are irreducible.

Note that all the above representations are orthogonal.

\subsection{Restriction to a Sylow $2$-Subgroup}\label{ressyl}

Set $\chi_0=1$. In this section, we write $\chi_{i, m}$ for the linear characters $\chi_i$ and $\sigma_{k,m}$ for the representations $\sigma_k$ of the group $D_{m}$.

 Assume $m=2^{n}l$ for some $n\in \mathbb{N}$ and $l$ odd. $D_{2^n}$, with usual inclusion, is a Sylow $2$-subgroup of $D_{m}$. It is known \cite[Chapter 1, Section 4]{VS} that 

\begin{equation}\label{cohomD2n} H^*(D_{2^n})=\begin{cases}
\z/2\z[v]& n=0,\\
\z/2\z[v_1,v_2]& n=1,\\
\z/2\z[x,y,w]/(y^2+xy)& n\geq 2.
\end{cases}\end{equation}
where $v=w_1(\chi_{s,1})$, $v_1=w_1(\chi_{r,2})$, $v_2=w_1(\chi_{s,2})$, and $x= w_1(\chi_{s,2^n}), $ $y=w_1(\chi_{r,2^n})$ and $w=w_2(\sigma_{1,2^n})$ for $n\geq 2$.

Consider the inclusion $\iota: D_{2^n} \to D_{m}$. Then, the restriction map\begin{equation}\label{incl}\iota^*:H^*(D_{m})\to H^*(D_{2^n})\end{equation}
is an isomorphism \cite[Theorem 4.6]{VS}. To understand $\iota^*$, the restrictions of representations of $D_{m}$ to $D_{2^n}$ are of interest. Below we define elements of $H^*(D_{m})$ which map to the generators of $H^*(D_{2^n})$ under $\iota^*$.
This is done in three cases.

For $m$ odd, the linear character $\chi_{s,m}$ restricted to $D_{1}$ is the $\sgn$ representation. Then, the functoriality of SWCs gives
\begin{align*}
\iota^*(w_1(\chi_{s,m}))&=w_1(\sgn)\\
&=v.
\end{align*}
Without ambiguity, we simply write \begin{equation}\label{v}w_1(\chi_{s,m})=v.\end{equation}

Suppose now that $m = 2 \pmod 4$. The restriction of $\chi_{i,m}$ to $D_2$ is $\chi_{i,2}$ for each $i=0,r,s,rs$. Therefore, 
$\iota^*(w_1(\chi_{r,m}))=v_1,$ and
$\iota^*(w_1(\chi_{s,m}))=v_2$
and we write \begin{equation}\label{v1v2} 
\begin{split}w_1(\chi_{r,m})&=v_1, \text{ and }\\
w_1(\chi_{s,m})&=v_2.
\end{split}
\end{equation}

Finally, suppose $m = 0 \pmod 4$. In this case, $\chi_{i,m}$ restricts to $ \chi_{i, 2^n}$ and $\sigma_{1,m}$ restricts to $\sigma_{1,2^n}$. Again, since $\iota^*$ is an isomorphism, we write

\begin{equation}\label{xyw} 
\begin{split}
w_1(\chi_{s,m})&=x\\
w_1(\chi_{r,m})&=y\\
w_1(\sigma_{1,m})&=w.
\end{split}
\end{equation}

With notation given by \eqref{v},\eqref{v1v2} and \eqref{xyw}, an equality
similar to \eqref{cohomD2n} holds for $H^*(D_{m})$. Henceforth, we use the same set of notations for the generators of group cohomology of $D_m$ and $D_{2^n}$.

\subsection{\texorpdfstring{Extensions of \( \Or(V) \)}{Extensions of Or(V)}} \label{spindef}
Consider an orthogonal representation $\pi:G\to \Or(V)$. Here, $V$ is a finite-dimensional (complex) vector space, with a nondegenerate quadratic form $Q$.   Denote a classifying space of $\Or(V)$ by BO$(V)$. An explicit one-to-one correspondence of the elements of $H^2($BO$(V),\z/2\z)$ with the extensions of $\Or(V)$ by $\z/2\z$ is known (see, for instance, \cite{SJext}). There are four mutually inequivalent such extensions of $\Or(V)$, namely $\Or(V)\times \z/2\z$, $\widetilde\Or(V)$, $\Pin^+(V)$ and $\Pin^-(V)$.

 The first one is easy to describe. For $\widetilde\Or(V)$, consider the determinant map from $\Or(V)$, which surjects onto $C_2 = \{\pm 1\}$. Of course, squaring the cyclic group $C_4 < \mb C^\times$, generated by the imaginary unit $i$, gives a $\z/2\z$-cover of $C_2$. 
Write $\widetilde \Or(V) \to \Or(V)$ for the pullback:
\beq 
\xymatrix{
	\widetilde \Or(V)   \ar[d]   \ar[r]  &  C_4  \ar[d]^{z \mapsto z^2}\\
	\Or(V)  \ar[r]^{\det} & C_2}
\eeq 
Thus $\widetilde \Or(V)$ is the subgroup of  pairs $(g,z) \in \Or(V) \times C_4$ with $\det g=z^2$.

For a vector space $(V,Q)$, the Clifford algebra $C(V)$ is the quotient of the tensor algebra $T(V)$ by the two-sided ideal generated by the set
\beq
\{ v\otimes v -Q(v): v\in V\}.
\eeq
It contains $V$ as a subspace.  Write $C(V)^\times$ for the group of invertible elements of $C(V)$. We say $u \in V$ is a \emph{unit vector} if $Q(u)=1$, and an \emph{antiunit vector}, if $Q(u)=-1$. 

 Write $\Pin^-(V)$ for the subgroup of $C(V)^\times$ generated by the antiunit vectors in $V \subset C(V)$, and
	$\Pin^+(V)$ for the subgroup generated by the unit vectors. This group $\Pin^-(V)$ agrees with the one described in \cite[Chapter 20]{FH}, and $\Pin^+(V)$ with \cite[Appendix 1]{froh}.

 Moreover, the following cohomological criterion is known (see \cite[Corollary 12.2.1]{SJext} for instance): The representation $(\pi,V)$ lifts to:

\begin{enumerate}
\item $\widetilde\Or(V)$ iff $w_1(\pi)^2=0$,
\item $\Pin^+(V)$ iff $w_2(\pi)=0$, and
\item $\Pin^-(V)$ iff $w_2(\pi)+w_1(\pi)^2=0$.
\end{enumerate}

\section{The case of $C_2$ and $C_2 \times C_2$}\label{small groups}

 Let $C_2=\{\pm 1\}$. Then, $D_1$ is simply $C_2$ and $D_2$ is the Klein-$4$ group $C_2\times C_2$. Below are the formulas for their SWCs.

\subsection{The Cyclic group $C_2$}
Let $\pi$ be a representation of $C_2$. The total SWC of $\pi$ is known (see, for instance, in \cite[Lemma 2.5]{NSSL2}). With $v=w_1(\sgn)$, we have
\begin{equation}\label{swcC2}w(\pi)=(1+v)^{t_\pi},\end{equation}
where $t_\pi=\frac 12(\chi_\pi(1)-\chi_\pi(-1))$.

\subsection{The Klein-$4$ Group}\label{C2C2}
Consider the group $C_2\times C_2$ and the projection maps $\pr_i:C_2\times C_2\to C_2$ for $i=1,2$.
 By K\"unneth, we have
$$H^*(C_2\times C_2)\cong \z/2\z[v_1,v_2],$$
where $v_1=\pr^*_1(v)=w_1(\sgn\boxtimes 1)$ and $v_2=\pr^*_2(v)=w_1(1\boxtimes \sgn)$. Here $\boxtimes$ denotes the external tensor product.

\begin{prop}\label{C2C2SWC}
Let $\pi$ be a representation of $C_2\times C_2$. The total SWC of $\pi$ is
$$w(\pi)=(1+v_1)^{b_1}(1+v_2)^{b_2}(1+v_1+v_2)^{b_3},$$
where
\begin{align*}
b_1&=\frac14\left(\deg\pi-\chi_\pi(-1,1)+\chi_\pi(1,-1)-\chi_\pi(-1,-1)\right)\\
b_2&=\frac14\left(\deg\pi+\chi_\pi(-1,1)-\chi_\pi(1,-1)-\chi_\pi(-1,-1)\right)\\
b_3&=\frac14\left(\deg\pi-\chi_\pi(-1,1)-\chi_\pi(1,-1)+\chi_\pi(-1,-1)\right).
\end{align*}

\end{prop}

\begin{proof}

 Any representation $\pi$ of $C_2\times C_2$ has the form
$$\pi=b_01\oplus b_1(\sgn\boxtimes 1)\oplus b_2(1\boxtimes \sgn)\oplus b_3(\sgn\boxtimes\sgn).$$
Note all representations of $G$ are orthogonal.
By multiplicativity of SWCs, we have
$$w(\pi)=(1+v_1)^{b_1}(1+v_2)^{b_2}(1+v_1+v_2)^{b_3}.$$
To express $b_i$ in terms of character values, we use the following equations:
\begin{align*}
\chi_\pi(1,1)&=b_0+b_1+b_2+b_3\\
\chi_\pi(-1,1)&=b_0-b_1+b_2-b_3\\
\chi_\pi(1,-1)&=b_0+b_1-b_2-b_3\\
\chi_\pi(-1,-1)&=b_0-b_1-b_2+b_3.
\end{align*}
Solving these for $b_i$ completes the proof.

\end{proof}

\section{Main Results}\label{D2m} 

Recall $D_{2^n}$ sits in $D_m$ under the inclusion $\iota$, assuming $m=2^nl$ with $n\in\mathbb{N}$ and $l$ odd. As the SWCs are functorial and the restriction map $\iota^*$, in \eqref{incl}, is an isomorphism, it suffices to work with $D_{2^n}$ to find SWCs for $D_m$.

The total SWC for $D_1$ is given in Equation \eqref{swcC2}, and the same formula with $$v=w_1(\chi_s),\text { and }t_\pi=\frac12\left(\chi_\pi(1)-\chi_\pi(s)\right)$$ holds for $D_m$ when $m$ is odd, due to the isomorphism \eqref{incl}.

Let $r_c=r^{m/2}$. For dihedral group with $m= 2\pmod 4$, we can identify its detecting subgroup $D_2=\{1,r_c,s,sr_c\}$ with $C_2\times C_2$ as $r_c\leftrightarrow(-1,1)$, $s \leftrightarrow (1,-1)$. Then, Proposition \ref{C2C2SWC} translates to give the SWCs for these dihedral groups. For instance, in this case, with $\pi$ a representation of $D_m$, we have 
$$b_1=\frac14\left(\deg\pi-\chi_\pi(r_c)+\chi_\pi(s)-\chi_\pi(sr_c)\right).$$

\subsection{Detection}\label{detect} 
In this section, suppose $m$ is a multiple of $4$, and $G=D_{m}$. 

 We consider the following subgroups of $G$:
\begin{align*}
E_1&=\{1,s,r_c,sr_c\},\\
E_2&=\{1,rs,r_c,rsr_c\}.
\end{align*}
Both $E_1,E_2$ are isomorphic to the Klein $4$-group. Let $\as_1,\beta_1$ be the linear characters of $E_1$ given by
\begin{align*}
\as_1:(s,r_c)&\mapsto(-1,1)\\
 \beta_1:(s,r_c)&\mapsto(1,-1).
\end{align*}Then, $H^*(E_1)\cong \z/2\z[v_1,v_2],$ with $v_1=w_1(\as_1),$
$v_2=w_1(\beta_1)$. 
Similarly, we consider the linear characters $\as_2,\beta_2$ of $E_2$ defined by,
\begin{align*} 
\as_2:(rs,r_c)&\mapsto(-1,1)\\
  \beta_2:(rs,r_c)&\mapsto(1,-1)\end{align*}
such that $H^*(E_2)\cong \z/2\z[u_1,u_2]$ with $u_1=w_1(\as_2)$, $u_2=w_1(\beta_2)$.\\

The following detection can be found in \cite[Proposition 3.3, Page 322]{FP}.  Although the result itself is accurate, an error is found within the accompanying proof provided in this book. We rectify the issue in the proof below. 

\begin{prop}\label{e1e2} When $m$ is a multiple of $4$, the subgroups $E_1$, $E_2$ together detect the mod $2$ cohomology of $D_{m}$. 

\end{prop}
Proving this requires a lemma:

\begin{lemma}\label{detlem}
 Let $`d$' be a non-negative integer, and $\mathcal{P}_d=\{(i,j) : i,j\in  \Z_{\geq 0},\;  i + 2j = d \}$. Then, the set $S_d = \{a^ib^j(a+b)^j : (i,j) \in \mathcal P_d\}$ in $\Z/2\z[a,b]$ is linearly independent.	
\end{lemma}	

\begin{proof}
	The set $\mathcal P_d$ can be enumerated as follows:
	\begin{equation*} 
		(d,0), (d-2,1),\ldots,(d-2k,k), \ldots, 
		\begin{cases}
			(0,d/2), &\quad\text{ when } d \text{ is even}\\
			(1,(d-1)/2),&\quad\text{ when } d \text{ is odd}.
		\end{cases}
	\end{equation*}
For a polynomial $p\in\z/2\z[a,b]$, let $\deg_a(p)$ be the highest degree of $`a$' appearing in $p$. If we put $q_{ij}(a,b) = a^ib^j(a+b)^j$, then $\deg_a(q_{ij})=i+j$.

 Now suppose $S_d$ is linearly dependent. That is, there are $c_k$, not all zero, such that
\begin{equation}\label{ck}\sum\limits_{k=0}^{\lfloor d/2 \rfloor} c_k q_{d-2k,k}= 0. \end{equation}

Let $\ell$ be the least integer with $c_{\ell}=1$. For $k>\ell$, we have $$\deg_a(q_{d-2k,k})=d-k<d-\ell.$$
But then condition \eqref{ck} forces $c_\ell$ to be zero, which is a contradiction.

\end{proof}

\begin{proof}[Proof of Proposition \ref{e1e2}]
Consider the restriction map
$$\res^*:H^*(G)\to H^*(E_1)\oplus H^*(E_2).$$

We first find the images of $x,y,w$ under $\res^*$. An easy calculation shows
\begin{equation}\label{chires}
\res^G_{E_1}\chi_r=1,\;\res^G_{E_2}\chi_r=\as_2, \text{ and }\res^G_{E_i}\chi_s=\as_i \end{equation} for $i=1,2$. This, together with \eqref{cohomD2n}, 
\eqref{incl}, implies
 $$\res^*(x)=(v_1,u_1)\quad,\quad\res^*(y)=(0,u_1).
$$

For $\res^*(w)$, consider the standard representation $\sigma$ of $G$. A simple eigenvalue calculation shows $\res^G_{E_1}\sigma$ is equivalent to the representation mapping
$$r_c\mapsto \begin{pmatrix}
-1&&0\\
0&&-1
\end{pmatrix}\quad,\quad s\mapsto \begin{pmatrix}
-1&&0\\
0&&1
\end{pmatrix}.$$
This gives $\res^G_{E_1}\sigma=\beta_1\oplus(\as_1\otimes\beta_1)$, implying
\begin{align*}
w(\res^G_{E_1}\sigma)&=w(\beta_1\oplus(\as_1\otimes\beta_1))\\
&=(1+v_2)(1+v_1+v_2)\\
&=1+v_1+v_2(v_1+v_2).
\end{align*}
The above equality uses $w(\as_1\otimes\beta_1)=1+w_1(\as_1)+w_1(\beta_1).$
Similarly, we have $w(\res^G_{E_2}\sigma)=1+u_1+u_2(u_1+u_2).$ Therefore, $\res^*$ maps
\begin{equation}\label{resxyw}
\begin{split}
x&\mapsto(v_1,u_1)\\
y&\mapsto (0,u_1)\\
w&\mapsto (v_2^2+v_1v_2,u_2^2+u_1u_2).
\end{split}
\end{equation}
To prove $\res^*$ is injective, consider an arbitrary element $g=\sum\limits_{i,j,k\geq 0} a_{ijk}x^iy^jw^k\in H^*(G)$. The summand might have terms of the form $y^jw^k$ for $i=0$ and $x^iw^k$ for $j=0$. But for terms with $i,j>0$, we use $y^2+xy=0$ in $H^*(G)$ to simplify $x^iy^jw^k$:
\begin{align*}
x^iy^jw^k&=x^{i-1}(xy)y^{j-1}w^k\\
&=x^{i-1}(y^2)y^{j-1}w^k\\
&=x^{i-1}y^{j+1}w^k\\
&=y^{i+j}w^k.
\end{align*}
Thus, we can write
$$g=\sum_{i>0,\;j\geq 0}a_{ij}x^iw^j+\sum_{k>0,\;l\geq 0}b_{k,l}y^kw^l+\sum_{t\geq 0}c_tw^t.$$

Set $\delta_d=\begin{cases}
0& d\text{ is odd}\\
1& d\text{ is even}
\end{cases},$ and  $\mc P^+_0=\emptyset$. For a positive integer $d$, let $\mc P^+_d=\{(i,j)\in \z_{>0}\times\z_{\geq0}:i+2j=d\},$ a subset of $\mathcal P_d$. Then,
\begin{align*}
\res^*(g)&=\res^*\left(\sum_{d> 0}\sum_{(i,j)\in P_d}a_{ij}x^iw^j+\sum_{d> 0}\sum_{(k,l)\in P_d}b_{k,l}y^kw^l+\sum_{t\geq 0}c_tw^t\right)\\
&=\sum_{d\geq0}\left(\sum_{(i,j)\in P_d}a_{ij}\res^*(x^iw^j)+\sum_{(k,l)\in P_d}b_{k,l}\res^*(y^kw^l)+\delta_dc_{d/2}w^{d/2}\right),
\end{align*}
where for each $d$, the term inside the parenthesis belongs to $H^d(G)$. Suppose $\res^*(g)=0$. Since $H^*(G)$ is a graded ring, this is equivalent to saying for each $d\geq 0$,
$$\sum_{(i,j)\in P_d}a_{ij}\res^*(x^iw^j)+\sum_{(k,l)\in P_d}b_{k,l}\res^*(y^kw^l)+\delta_dc_{d/2}\res^*(w^{d/2})=0.$$

Thus, for injectivity of $\res^*$, it is enough to prove that the set  $$\{\res^*(x^iw^j), \res^*(y^kw^l),\delta_d\res^*(w^{d/2}):(i,j),(k,l)\in \mc P^+_d\}$$ is linearly independent for all $d\in \z_{\geq 0}.$ We have
\begin{align*}
\res^*(x^iw^j)&=(v_1^iv_2^j(v_1+v_2)^j,u_1^iu_2^j(u_1+u_2)^j)\text{ for } i,j\geq 0\\
\res^*(y^kw^l)&=(0,u_1^ku_2^l(u_1+u_2)^l)\text{  for } k>0, \:l\geq 0.
\end{align*}
This further reduces our problem to showing
$$\{(v_1^iv_2^j(v_1+v_2)^j,0), (0,u_1^ku_2^l(u_1+u_2)^l), \delta_d(v_2^{d/2}(v_1+v_2)^{d/2},u_2^{d/2}(u_1+u_2)^{d/2}):(i,j),(k,l)\in \mc P^+_d\}$$
is linearly independent for all $d\in \z_{\geq 0}$, which in turn follows from Lemma \ref{detlem}.

\end{proof}
 We now use Proposition \ref{e1e2} to determine SWCs for $G=D_m$.

\subsection{Formula for SWCs}
  Let $m=0\pmod 4$, and $\pi$ be a representation of $G = D_m$. It has the form\begin{equation}\label{fpi}
\pi=n_01\oplus n_s\chi_s\oplus n_r \chi_r\oplus n_{rs}\chi_{rs}\oplus\bigoplus_{i=1}^{m/2-1}d_i\sigma_i,
\end{equation}
where $n_0$, $n_s$, $n_r$, $n_{rs}$, $d_i$ are non-negative integers. Define $$d_e=\sum\limits_{i\text{ even }}d_i\quad, \quad d_o=\sum\limits_{i\text{ odd}}d_i.$$

To determine the total SWC $w(\pi)$, we first find SWCs for $\sigma_k$. Consider the subgroups $E_1,E_2$ with linear characters $\as_i,\beta_i$ as defined in Section \ref{detect}. From an eigenvalue calculation, we observe that
\begin{equation}\label{sigres}\res^G_{E_i}\sigma_k=\begin{cases}
\beta_i\oplus (\as_i\otimes\beta_i),& \text{when $k$ is odd}\\
1\oplus\as_i,& \text{when $k$ is even}.
\end{cases}\end{equation}
When $k$ is odd,
\begin{align*}
\res^*(w(\sigma_k))&=(w(\res^G_{E_1}\sigma_k), w(\res^G_{E_2}\sigma_k))\\
&=((1+v_2)(1+v_1+v_2),(1+u_2)(1+u_1+u_2))\\
&=(1+v_1+v_2(v_1+v_2),1+u_1+u_2(u_1+u_2))\\
&=(1,1)+(v_1,u_1)+(v_2^2+v_1v_2,u_2^2+u_1u_2).
\end{align*}
Thus, $w(\sigma_k)=1+x+w$, due to \eqref{resxyw}. Similarly, when $k$ is even, \begin{align*}
\res^*(w(\sigma_k))&=(1+v_1,1+u_1)\\
&=(1,1)+(v_1,u_1)
\end{align*}
which gives $w(\sigma_k)=1+x$ in this case.\\

Now we prove our main result about the total SWC $w(\pi)$:

\begin{proof}[Proof of Theorem \ref{main}]
Let $\pi$ be as in \eqref{fpi}.
By multiplicativity of SWCs, we have
\begin{align*}
w(\pi)&=(1+x)^{n_s+d_e}(1+y)^{n_r}(1+x+y)^{n_{rs}}(1+x+w)^{d_o}.
\end{align*}
As $(1+y)(1+x+y)=1+x$ in $H^*(G)$, the above formula becomes
$$w(\pi)=(1+y)^{a_\pi}(1+x+y)^{b_\pi}(1+x+w)^{c_\pi}.$$
with  $a_\pi=n_s+n_r+d_e$, $b_\pi=n_s+n_{rs}+d_e$, $c_\pi=d_o$.\\

To determine the character formulas for $a_\pi,b_\pi$ and $c_\pi$, we restrict $\pi$ to $E_i$:
\begin{equation}\label{rese1e2}
\begin{split}
\res^G_{E_1}\pi&=(n_0+n_r+d_e)1\oplus (n_s+n_{rs}+d_e)\as_1\oplus d_o\beta_1\oplus d_o(\as_1\otimes\beta_1),\\
\res^G_{E_2}\pi&=(n_0+n_{rs}+d_e)1\oplus (n_s+n_r+d_e)\as_2\oplus d_o\beta_2\oplus d_o(\as_2\otimes\beta_2).
\end{split}
\end{equation} 
This comes from \eqref{chires} and \eqref{sigres}. For $\psi$  irreducible, let mult$(\psi,\varphi)$ denote the multiplicity of $\psi$ in $\varphi$. Note that 
\begin{align*}
a_\pi&=\mult(\as_2,\res^G_{E_2}\pi)\\
b_\pi&=\mult(\as_1,\res^G_{E_1}\pi)\\
c_\pi&=\mult(\as_i\otimes\beta_i,\res^G_{E_i}\pi).
\end{align*}
We identify $E_1$ with Klein-$4$ group by $s\leftrightarrow (-1,1)$ and $r_c\leftrightarrow (1,-1)$. Similarly, for $E_2$, we identify $rs$ with $(-1,1)$ and $r_c$ again with $(1,-1)$.
This forces the identifications $\as_i\leftrightarrow\sgn\boxtimes 1$, $\beta_i\leftrightarrow 1\boxtimes \sgn$, and so on.

Using the character formulas from Proposition \ref{C2C2SWC} yields
\begin{align*}
a_\pi&=\frac14(\chi_\pi(1)-\chi_\pi(rs)+\chi_\pi(r_c)-\chi_\pi(rsr_c)),\\
b_\pi&=\frac14(\chi_\pi(1)-\chi_\pi(s)+\chi_\pi(r_c)-\chi_\pi(sr_c)),\\
c_\pi&=\frac14(\chi_\pi(1)-\chi_\pi(s)-\chi_\pi(r_c)+\chi_\pi(sr_c)).
\end{align*}

The elements $s$ and $sr_c$ are conjugates in $G$, and so are $rs$ and $rsr_c$. Therefore, we have the desired formulas by using $\chi_\pi(s)=\chi_\pi(sr_c)$, and  $\chi_\pi(rs)=\chi_\pi(rsr_c)$.

\end{proof}
\begin{ex} Let $\reg(G)$ be the regular representation of $G$. When $m$ is a multiple of $4$, 
$$w(\reg(D_m))=\left(1+x^2+w+xw\right)^{m/2}.$$
\end{ex}
We now prove the corollaries to Theorem \ref{main}.

\subsection{Corollaries}\label{sectW(G)}

Since all representations are orthogonal for $G=D_m$, we have 
$$W(G)=
\langle w(\pi):\pi\text{ irreducible}\rangle\leqslant H^\bullet(G).$$

When $m$ is odd, the character $\chi_s$ has its total SWC $(1+v)$. The multiplicativity of SWCs, then, gives
$$W(D_m)=\{(1+v)^n:n\in \z\}.$$
Similarly, from Section \ref{C2C2}, it is straightforward that when $m= 2\pmod 4$, 
$$W(D_m)=\{(1+v_1)^a(1+v_2)^b(1+v_1+v_2)^c:a,b,c\in \z\}.$$

We now determine this subgroup for $G=D_m$ when $m= 0\pmod 4$:

\begin{proof}[Proof of Corollary \ref{W(G)iso}]
We define $\phi: \Z^3 \to W(G)$ by, $$\phi(a,b,c)=  (1+y)^a(1+x+y)^b(1+x+w)^c.$$ 
Recall there are representations $\chi_r$, $\chi_{rs}$, $\sigma$ of $G$ such that 
\begin{align*}
w(\chi_r)&=1+y,\\
w(\chi_{rs})&=1+x+y,\\
w(\sigma)&=1+x+w.
\end{align*}
The multiplicativity of SWCs then implies that $\phi$ is surjective.

 Since $\phi$ is a group homomorphism and one can express $v\in \z^3$ as a difference $v=v_1-v_2$, where $v_1,v_2\in \z^3_{\geq 0}$,  it is enough to prove injectivity on $\z^3_{\geq 0}$.


We first perform a change of variables in $H^*(D_m)$ (from Section \ref{ressyl}) by setting $z=x+y$ so that
$$H^*(D_{m})= \Z/2\Z[z,y,w]/(yz). $$ 
With this, $\phi$ becomes
$$  \phi(a,b,c)=(1+y)^a(1+z)^b(1+z+y+w)^c.$$


Assuming $a,b,c$ are non-negative, the degrees of $y,z,w$ in $\phi(a,b,c)$ are $a+c,\: b+c,\:c$ respectively.	
Therefore, if $\phi(a,b,c)= \phi(a',b',c')$ for non-negative triplets $(a,b,c)$, $(a',b',c')$, then by comparison of degrees of $y,z,w$, the triplets must be the same. 
\end{proof} 
As a consequence, we have:
\begin{proof}[Proof of Corollary \ref{wtriv}]
Recall from the proof of Theorem \ref{main} that $a_\pi=n_s+n_r+d_e$, $b_\pi=n_s+n_{rs}+d_e$, $c_\pi=d_o$.  In the proof of Corollary \ref{W(G)iso} above, since $\phi$ is an isomorphism, we have $w(\pi)=1$ if and only if $a_\pi=b_\pi=c_\pi=0$. Thus, the multiplicities $d_e,d_o,n_s,n_r,n_{rs}$ in Equation \eqref{fpi} are all zero, which completes the proof.
\end{proof}
\noindent A similar argument holds when $m$ is odd or $m= 2\pmod 4$.

\section{Top SWC}\label{TSWCsec}
In this section, we characterize representations $\pi$ with non-trivial $w_{\too}(\pi)$. When $m$ is a multiple of $4$, Corollary \ref{topswc} gives the top SWC, which we prove below.
\begin{proof}[Proof of Corollary \ref{topswc}]
 First, we suppose $w_{\too}(\pi)\neq 0.$ Recall $\pi$ has the form \eqref{fpi}. Clearly, irreducible representations with trivial top SWC must not appear in this decomposition. This implies $n_0=0$ as $w_{\too}(1)=0$, and $d_e=0$, as $w_{\too}(\sigma_{k})=0$ for all even $k$. Thus,
$$w(\pi)=(1+y)^{n_r}(1+x)^{n_s}(1+x+y)^{n_{rs}}(1+x+w)^{d_o}.$$

Suppose $n_r,n_{rs}$ are both non-zero, and WLOG $n_r<n_{rs}$. The relation $(1+y)(1+x+y)=1+x$ in $H^*(G)$ simplifies $w(\pi)$ to
$$w(\pi)=(1+x)^{n_r+n_s}(1+x+y)^{n_{rs}-n_r}(1+x+w)^{d_o}.$$

Here, the highest non-zero SWC is $w_h(\pi)=x^{n_r+n_s}(x+y)^{n_{rs}-n_r}w^{d_o}$ for $h=n_s+n_{rs}+2d_o$ which is not equal to $\deg \pi$. This is a contradiction. Therefore, either $n_r=0$ or $n_{rs}=0$. 

On the contrary, if $n_0=d_e=n_r=0$, then
$$w(\pi)=(1+x)^{n_s}(1+x+y)^{n_{rs}}(1+x+w)^{d_o},$$
giving $w_{\too}(\pi)=x^{n_s}(x+y)^{n_{rs}}w^{d_o}\neq 0.$ A similar argument holds when $n_0=d_e=n_{rs}=0$.

 From Equation \eqref{rese1e2} and the character formula for the multiplicity of $1$ in a representation, we obtain
\begin{align*}
n_0+n_r+d_e&=\frac14\sum_{g\in E_1}\chi_\pi(g),\\
n_0+n_{rs}+d_e&=\frac14\sum_{g\in E_2}\chi_\pi(g).
\end{align*}
It follows from these equations that the condition $n_0=d_e=n_r=0$ (or $n_0=d_e=n_{rs}=0$) is equivalent to having $\ell_\pi=0$ (or resp. $m_\pi=0$).

\end{proof}
Recall the mod $2$ cohomology of $D_m$ is detected by $D_1=\{1,s\}$, when $m$ is odd, and by $D_2=\{1,s,r_c,sr_c\}$ when $m= 2\pmod 4$. For these cases, we have:

\begin{prop}
Let $\pi$ be a representation of $D_m$. The top SWC of $\pi$ is non-zero if and only if 
\begin{roster}

\item  $\deg\pi+\chi_\pi(s)=0$ when $m$ is odd,
\item $\deg\pi+\chi_\pi(r_c)+\chi_\pi(s)+\chi_\pi(sr_c)=0$ when $m= 2\pmod 4$.

\end{roster}

\end{prop}


\begin{proof}
For $H$ a detecting subgroup, $w_{\too}(\pi)$ is non-zero iff $w_{\too}(\pi|_H)$ is non-zero. When $m $ is odd, the restriction of $\pi$ to $D_1$ is,
$$\pi|_H=a_01\oplus a_1\sgn.$$

Now, $w_{\too}(\pi|_{D_1})$ is non-zero iff $a_0=\mult(1,\pi|_{D_1})=\frac12\left(\deg\pi+\chi_\pi(s)\right)$ is zero.\\

Similarly, when $m= 2\pmod 4$, we have

$$
\pi|_{D_2}=b_{0}1\oplus\bigoplus_{\substack{\chi\neq 1\\ \chi\text{ irreducible} }}b_\chi\chi 
$$
with $b_0=\mult(1,\pi|_{D_2})=\frac14\sum\limits_{g\in D_2}\chi_\pi(g)$. Again, $w_{\too}(\pi|_{D_2})\neq 0$ iff $b_0=0$, proving our claim.
\end{proof}

\section{Lifting to the Extensions of O$(V)$}\label{liftOV}
Recall from Section \ref{spindef}, $\Or(V)$ has three non-trivial extensions by $\z/2\z$: $\widetilde\Or(V)$, $\Pin^+(V)$ and $\Pin ^-(V)$. Here, we give a criterion in terms of character values to see whether a representation $\pi$ of $D_m$ lifts to these extensions or not. Corollary \ref{spin} states this for the case when $m$ is multiple of $4$:
\begin{proof}[Proof of Corollary \ref{spin}]

From Theorem \ref{main}, we deduce
\begin{align*}
w_1(\pi)&=(b_\pi+c_\pi)x+(a_\pi+b_\pi)y,\\
w_2(\pi)&=\left[\binom{b_\pi}{2}+\binom{c_\pi}{2}+b_\pi c_\pi\right]x^2+\left[\binom{a_\pi}{2}+\binom{b_\pi}{2}\right]y^2+c_\pi w+c_\pi(a_\pi+b_\pi)xy.
\end{align*}
Recall $\pi$ lifts to $\widetilde\Or(V)$ iff $w_1(\pi)^2=0$, which in this case is iff $a_\pi+b_\pi=b_\pi+c_\pi=0\pmod 2$.\\
Similarly, $\pi$ lifts to $\Pin^+(V)$ iff $w_2(\pi)=0$. Clearly,
$c_\pi$ must be $0\pmod 2$. This simplifies $w_2(\pi)$ to give
\begin{align*}
w_2(\pi)&=\left[\binom{b_\pi}{2}+\binom{c_\pi}{2}\right]x^2+\left[\binom{a_\pi}{2}+\binom{b_\pi}{2}\right]y^2.
\end{align*}
 which gives the condition $$\binom{a_\pi}{2}+\binom{b_\pi}{2}=\binom{b_\pi}{2}+\binom{c_\pi}{2}=0\pmod 2.$$
 A representation $\pi$ is spinorial (or lifts to $\Pin^-(V)$) if and only if \begin{equation}\label{w2w1}w_2(\pi)=w_1(\pi)\cup w_1(\pi).\end{equation} 
Again, $c_\pi$ must be even and by comparing coefficients in \eqref{w2w1}, we obtain
\begin{align*}
\frac{a_\pi^2+a_\pi}{2}&=\frac{b_\pi^2+b_\pi}{2}=\frac{c_\pi^2+c_\pi}{2}\pmod 2,
\end{align*}
which completes the proof.
\end{proof}

\begin{prop}
Let $m$ be odd, and $\pi$ be a representation of $D_m$. Let $$t_\pi=\frac12(\chi_\pi(1)-\chi_\pi(s)).$$ Then, $\pi$ lifts to
\begin{roster}
\item $\widetilde \Or(V)$ iff  $t_\pi$ is even, 

\item $\Pin^+(V)$ iff $t_\pi= 0,1\pmod 4$, and

\item $\Pin^-(V)$ iff $t_\pi= 0,3\pmod 4$.
\end{roster}
\end{prop}
\begin{proof}
Recall  for $m$ odd, $$w(\pi)=(1+v)^{t_\pi},$$ giving
$w_1^2(\pi)=t_\pi v^2$, and $w_2(\pi)=\dbinom{t_\pi}{2}v^2$.\\
The proof then follows using the lifting criteria from Section \ref{spindef}.

\end{proof}
\begin{prop}
Let $m= 2\pmod 4$, and $\pi$ be a representation of $D_m$ with
\begin{align*}
e_\pi&=\frac14\left(\deg\pi-\chi_\pi(r_c)+\chi_\pi(s)-\chi_\pi(sr_c)\right),\\
f_\pi&=\frac14\left(\deg\pi+\chi_\pi(r_c)-\chi_\pi(s)-\chi_\pi(sr_c)\right),\\
g_\pi&=\frac14\left(\deg\pi-\chi_\pi(r_c)-\chi_\pi(s)+\chi_\pi(sr_c)\right).
\end{align*}
Then, $\pi$ lifts to
\begin{roster}
\item $\widetilde \Or(V)$ iff  $e_\pi=f_\pi=g_\pi\pmod 2$,\vspace{0.5mm}

\item $\Pin^+(V)$ iff$\dbinom{e_\pi+g_\pi}{2}=\dbinom{f_\pi+g_\pi}{2}= e_\pi g_\pi+f_\pi g_\pi+e_\pi f_\pi=0\pmod 2 $, and\vspace{0.5mm}

\item $\Pin^-(V)$ iff $\dbinom{e_\pi+g_\pi+1}{2}=\dbinom{f_\pi+g_\pi+1}{2}=\dbinom{g_\pi}{2}= e_\pi g_\pi+f_\pi g_\pi+e_\pi f_\pi=0\pmod 2 $.
\end{roster}
\end{prop}
\begin{proof} For ease of notation, write $e,f,g$ for $e_\pi$, $f_\pi$, $g_\pi$ above. From Proposition \ref{C2C2SWC}, we have
\begin{align*}
w(\pi)&=(1+v_1)^e(1+v_2)^f(1+v_1+v_2)^g\\
&=\left(1+ev_1+\binom{e}{2}v_1^2+\hdots\right)\left(1+fv_2+\binom{f}{2}v_2^2+\hdots\right)\\
&\hspace{40mm}\left(1+g(v_1+v_2)+\binom{g}{2}(v_1^2+v_2^2)+\hdots\right)
\end{align*}
This gives,
\begin{align*}
w_1(\pi)&=(e+g)v_1+(f+g)v_2\\
w_2(\pi)&=\left(\binom{e}{2}+eg+\binom{f}{2}\right)v_1^2+\left(\binom{f}{2}+fg+\binom{g}{2}\right)v_2^2+(ef+fg+eg)v_1v_2\\
&=\binom{e+g}{2}v_1^2+\binom{f+g}{2}v_1^2+(ef+fg+eg)v_1v_2
\end{align*}
Again, the proof is similar to that of Corollary \ref{spin} using lifting criteria in Section \ref{spindef}.\end{proof}

\begin{ex}
Let $m= 0\pmod 4$, and $G=D_m$.
\begin{enumerate}
 \item None of the non-trivial irreducible representations ($\pi,V$) of $G$ lift to $\widetilde\Or(V)$ and $\Pin^-(V)$.
\item All linear characters of $G$ have a lift to $\Pin^+(V)$. In fact,  $\pi=\sigma_k$ with $k$ odd are the only irreducible representations without a lift to $\Pin^+(V)$.
\item The representation $2\chi_s=\chi_s\oplus \chi_s$ lifts to $\widetilde\Or(V)$, but not to $\Pin^-(V)$. In contrast,  $4\chi_s$  lifts to all the extensions of $\Or(V)$ by $\z/2\z$. 
Also, the regular representation $\reg(G)$ lifts to all these extensions.
\end{enumerate}
\end{ex}

\section{Spinoriality for Products of Dihedral Groups}
Let $G, G'$ be finite groups. Let $\pi, \pi'$ be representations of $G$ and $G'$ respectively with $\deg \pi=d$, $\deg\pi'=d'$. Then $\Pi=\pi\boxtimes\pi'$, the external tensor product of $\pi$ and $\pi'$, is an orthogonal representation of $G\times G'$. From \cite[Section 7.1]{Ganguly}, such a representation is spinorial if and only if the following elements in $H^2(G\times G')$ vanish: 
\begin{enumerate}
\item\label{1} $d'w_2(\pi)+\binom{d'+1}{2}w_1(\pi)\cup w_1(\pi)$,
\item\label{2} $(dd'+1)w_1(\pi)\otimes w_1(\pi')$,
\item\label{3} $dw_2(\pi')+\binom{d+1}{2}w_1(\pi')\cup w_1(\pi')$.
\end{enumerate}
\begin{lemma}\label{cond}
The condition (\ref{1}) above is equivalent to the restriction of $\Pi$ to  $G\times 1$ being spinorial.
\end{lemma}
\begin{proof}
We have the restriction $\Pi|_{G\times 1}=d'\pi$, and its total SWC is
\begin{align*}
w(d'\pi)&=w(\pi)^{d'}\\
&=(1+w_1(\pi)+w_2(\pi)+\hdots)^{d'}\\
&=1+d'w_1(\pi)+d'w_2(\pi)+\binom{d'}{2}w_1(\pi)\cup w_1(\pi)+\hdots
\end{align*}
Now, $d'\pi$ is spinorial if and only if 
$$d'w_2(\pi)+\left(\binom{d'}{2}+d'^2\right) w_1(\pi)\cup w_1(\pi)\in H^2(G) $$
vanishes. Moreover, $$\binom{d'}{2}+d'^2=_2\binom{d'}{2}+d'=_2\binom{d'+1}{2},$$
completing the proof.

\end{proof}
Similarly, condition (\ref{3}) is same as the spinoriality of $\Pi|_{1\times G'}.$ Thus, the conditions (\ref{1}), (\ref{2}), (\ref{3}) along with Lemma \ref{cond} give Theorem \ref{prodDm}. A straightforward calculation with $G= D_m$ and $G' = D_{m'}$ yields a proof of Corollary \ref{aspin}.

\begin{ex}Suppose $m,m'$ both are multiples of $4$.
Let $k$ be even, and consider the representation of $D_{m} \times D_{m'}$ given by $$\Pi=(\sigma_k\oplus\sigma_k)\boxtimes(\sigma_k\oplus\sigma_k).$$
Applying Theorem \ref{prodDm} shows this representation is spinorial.
\end{ex}

\bibliographystyle{alpha}
\bibliography{mybib}
\vspace{10mm}

\end{document}